\documentclass[11pt,english,oneside]{amsart}

\usepackage[breaklinks,colorlinks,linkcolor=blue,citecolor=blue,urlcolor=blue]{hyperref}
\usepackage[T1]{fontenc}
\allowdisplaybreaks[4]
\usepackage[latin9]{inputenc}
\usepackage{geometry}
\usepackage{indentfirst}
\usepackage{amssymb}
\usepackage{color}

\usepackage{booktabs}
\usepackage{array, caption, threeparttable}
\usepackage[font=small,labelfont=bf,labelsep=none]{caption}
\usepackage{graphicx,color}
\usepackage{amsmath, amssymb, graphics}

\usepackage{graphicx}

\makeatletter
\numberwithin{equation}{section} 
\numberwithin{figure}{section} 
  \@ifundefined{theoremstyle}{\usepackage{amsthm}}{}
  \theoremstyle{plain}
  
  \theoremstyle{plain}
  
  \theoremstyle{plain}
  
  \theoremstyle{remark}
  
  \theoremstyle{remark}
  
  \theoremstyle{plain}
  
\usepackage{geometry}

\smallskip
\def\<{{\langle }}
\def\>{{\rangle }}

\makeatother

\usepackage{babel}

\smallskip
\def\<{{\langle }}
\def\>{{\rangle }}

\makeatother

\makeatletter
\newcommand{\rmnum}[1]{\romannumeral #1}
\newcommand{\Rmnum}[1]{\expandafter\@slowromancap\romannumeral #1@}
\makeatother

\usepackage[mathscr]{eucal}
\usepackage{amssymb}
\usepackage{latexsym}
\usepackage{amsthm}
\theoremstyle{plain}

\newtheorem{theorem}{Theorem}[section]
\newtheorem{lemma}[theorem]{Lemma}
\newtheorem{proposition}[theorem]{Proposition}

\newtheorem{remark}[theorem]{Remark}

\newtheorem{conjecture}[theorem]{Conjecture}

\makeatletter
\@addtoreset{equation}{section}

\usepackage{color}

\usepackage{graphicx,color}

\title[rigidity of $4$-dimensional complete self-shrinkers in $\mathbb{R}^{5}$]
{rigidity of $4$-dimensional complete self-shrinkers in $\mathbb{R}^{5}$}
\author{Chengyang Yi}
\address{Chengyang Yi \\  School of Mathematical Sciences, East China Normal University,
500 Dongchuan Road, Shanghai 200241, P. R. of China, E-mail address: 52195500013@stu.ecnu.edu.cn. }

\begin{document}
\maketitle

\begin{abstract}
We show that any $4$-dimensional complete self-shrinker in $\mathbb{R}^{5}$ with constant squared norm $S$ of the second fundamental form, $f_{3}=0$ and constant $f_{4}$ is isometric to $\mathbb{R}^{4}$, where $h_{ij}$ are components of the second fundamental form, $S=\sum h_{ij}^{2}$, $f_{3}=\sum h_{ij}h_{jk}h_{ki}$ and $f_{4}=\sum h_{ij}h_{jk}h_{kl}h_{li}$. As an application, we obtain a classification result.
\end{abstract}

\footnotetext{This work was partially supported by Science and Technology Commission of Shanghai Municipality (No. 22DZ2229014) and the National Natural Science Foundation of China (Grant No. 12271163). The research is supported by Shanghai Key Laboratory of PMMP.
}

\section{Introduction}
A family of smoothly immersed submanifolds $X(t):M^{n}\rightarrow \mathbb{R}^{n+p}$ for $t\in [0,T)$ satisfies mean curvature flow if
\begin{equation}
\frac{\partial X(t)}{\partial t}=\overrightarrow{H}(t),
\end{equation}
where $\overrightarrow{H}(t)$ denotes mean curvature vector of submanifold $X(t):M^{n}\rightarrow \mathbb{R}^{n+p}$. An immersion $X:M^{n}\rightarrow \mathbb{R}^{n+p}$ is called a self-shrinker if it satisfies
\begin{equation}
\overrightarrow{H}+X^{\perp}=0,
\end{equation}
 where $\perp$ is the projection onto the normal bundle of $M$. Actually, self-shrinkers correspond to self-shrinking solutions to the mean curvature flow. Moreover, they play an important role in the study on singularities of the mean curvature flow. One can refer to \cite{9,12,13}.

 It's very important to classify complete self-shrinkers. There are many results about the classification of self-shrinkers with co-dimension one, i.e. $p=1$. For $n=1$, Abresch and Langer \cite{1} gave a complete classification of closed self-shrinkers. Now, these curves are so-called Abresch-Langer curves. In higher dimension, Huisken \cite{12,13} proved that complete $n$-dimensional self-shrinkers in $\mathbb{R}^{n+1}$ with non-negative mean curvature, bounded $|A|$ and polynomial volume growth are $\Gamma \times \mathbb{R}^{n-1}$ or $S^{m}(\sqrt{m})\times\mathbb{R}^{n-m}(0\leq m \leq n)$, where $|A|$ is the norm of the second fundamental form and $\Gamma$ is a Abresch-Langer curve. Afterwards, Colding and Minicozzi \cite{9} removed the assumption that $|A|$ is bounded. In \cite{14}, Le and Sesum proved that any $n$-dimensional complete self-shrinker with polynomial volume growth in $\mathbb{R}^{n+1}$ must be a hyperplane $\mathbb{R}^{n}$ if $|A|<1$. Cao and Li \cite{2} proved that any complete self-shrinker with polynomial volume growth for arbitrary co-dimension must be a generalized cylinder provided $|A|\leq 1$. The above results show that there is a gap phenomena (so-called first gap) for $|A|$. As for the study of the second gap, we refer the readers to \cite{8,16}. It is very similar to the study of minimal submanifolds in the unit sphere (see below). So it's also interesting to classify complete self-shrinkers with constant squared norm of the second fundamental form. The following conjecture is well-known and very important.
\begin{conjecture}
An $n$-dimensional complete self-shrinker $X:M^{n}\rightarrow \mathbb{R}^{n+1}$ with constant squared norm of the second fundamental form is isometric to one of $\mathbb{R}^{n}$, $S^{k}(\sqrt{k})\times\mathbb{R}^{n-k}(0\leq k \leq n-1)$ and $S^{n}(\sqrt{n})$.
\end{conjecture}
In the case of dimension $2$, Ding and Xin \cite{11} confirmed the conjecture under the assumption of polynomial volume growth. By using of the generalized maximum principle for self-shrinkers \cite{7} (one can see Lemma 2.4 in Section 2), Cheng and Ogata \cite{6} have solved the above conjecture for $n=2$ affirmatively. But the conjecture is very difficult for the higher dimension $n$. For $n=3$, recently, Cheng, Li and Wei \cite{4,5} have proved the above conjecture is true under the assumption that $f_{3}$ or $f_{4}$ is constant.

It is necessary to mention Chern's conjecture for minimal hypersurfaces in the unit sphere which is closely related with the above conjecture.
\begin{conjecture}
Let $\mathscr{A}$ be the value of the squared norm of the second fundamental forms for $n$-dimensional closed minimal hypersurfaces in the unit sphere $S^{n+1}$ with constant scalar curvature, then the set of $\mathscr{A}$ should be discrete.
\end{conjecture}
\begin{remark}
In the above conjecture, it's trivial for $n=2$. By Gauss equation, the condition of scalar curvature is equivalent to constant squared norm of the second fundamental form.
\end{remark}
For the dimension $3$, in 1993, Chang \cite{3} have solved Conjecture 1.2 (even for constant mean curvature hypersurfaces). It is very difficult for $n=4$. In 2017, Deng, Gu and Wei \cite{10} gave a complete classification result of closed $4$-dimension minimal Willmore hypersurfaces (equivalent to $f_{3}=0$) in $S^{5}$ with constant scalar curvature. The condition of $f_{3}=0$ gives some symmetries of the principle curvatures (see Proposition 2.3 in Section 2).

Inspired by \cite{10}, we consider Conjecture 1.1 in the case of $f_{3}=0$ for $n=4$, but it's still hard. Under the assumption that $f_{4}$ is constant, we have the following result:
\begin{theorem}	\label{th:01}
Let $X:M^{4}\rightarrow \mathbb{R}^{5}$ be a $4$-dimensional complete self-shrinker in $\mathbb{R}^{5}$ with constant squared norm $S$ of the second fundamental form, $f_{3}=0$ and constant $f_{4}$. Then $X:M^{4}\rightarrow \mathbb{R}^{5}$ is isometric to $\mathbb{R}^{4}$.
\end{theorem}
As an application, we obtain a classification result:
\begin{theorem}	\label{th:01}
Let $X:M^{4}\rightarrow \mathbb{R}^{5}$ be a $4$-dimensional complete self-shrinker in $\mathbb{R}^{5}$ with constant squared norm $S$ of the second fundamental form, $f_{3}=\frac{3}{4}\overline{H}S-\frac{\overline{H}^{3}}{8}$ and constant $f_{4}$, where $\overline{H}=\inf |H|$. Then $X:M^{4}\rightarrow \mathbb{R}^{5}$ is isometric to one of
\begin{enumerate}
  \item $\mathbb{R}^{4}$,
  \item $S^{2}(\sqrt{2})\times \mathbb{R}^{2}$,
  \item $S^{4}(2)$.
\end{enumerate}
\end{theorem}

This paper is organized as follows. In Section 2, we will give basic concepts, formulas and important properties for self-shrinkers. Our proof of main theorem, i.e. Theorem 1.5, will be divided to two cases, the case when $H$ vanishes somewhere and the case when $H$ doesn't change sign. In Section 3, we will deal with the former case since the latter is similar to the former. In Section 4, we will give the proofs of Theorem 1.4 and Theorem 1.5.

\section{Preliminaries}

Let $X:M^{n}\rightarrow \mathbb{R}^{n+1}$ be a $n$-dimensional complete self-shrinker in $\mathbb{R}^{n+1}$. Choose a local orthonormal frame field $\{e_{A}\}_{A=1}^{n+1}$ along $M$ in $\mathbb{R}^{n+1}$ with  dual coframe field $\{\omega_{A}\}_{A=1}^{n+1}$, such that $e_{1}$, ... , $e_{n}$ are tangent to $M$ and $e_{n+1}$ are normal to $M$.  From now on, we use the following conventions on the ranges of indices,
\begin{equation*}
  1\leq i,j,k,l,m,r\leq n.
\end{equation*}
$\sum_{i}$ means taking summation from $1$ to $n$ for $i$. Denote $h_{ij}$, $H$ and $S$ as components of the second fundamental form $h$, the mean curvature and the squared norm of the second fundamental form respectively. Thus
\begin{equation*}
  h=\sum_{i,j}h_{ij}\omega_{i}\otimes\omega_{j},\ \ H=\sum_{i}h_{ii},\ \ S=\sum_{i,j}h_{ij}^{2}.
\end{equation*}
Let $f_{k}$ denote the smooth function on $M$ given by $\mathrm{tr}\left(h^{k}\right)$ for $k\geq3$. In particular,
\begin{equation*}
  f_{3}=\mathrm{tr}\left(h\circ h\circ h\right)=\sum_{i,j,k} h_{ij}h_{jk}h_{ki},\ \ f_{4}=\mathrm{tr}\left(h\circ h\circ h\circ h\right)=\sum_{i,j,k,l} h_{ij}h_{jk}h_{kl}h_{li}.
\end{equation*}
Let $R_{ijkl}$ be the components of the curvature tensor of $M$. We have the following Gauss equations:
\begin{equation}
R_{ijkl}=h_{ik}h_{jl}-h_{il}h_{jk}.
\end{equation}
Let $h_{ijk}$, $h_{ijkl}$ and $h_{ijklp}$ be the components of the first, the second and the third covariant derivatives of
the second fundamental form. We have the following Codazzi equations and  Ricci identities:
\begin{equation}
h_{ijk}=h_{jik}=h_{jki},
\end{equation}
\begin{equation}
h_{ijkl}-h_{ijlk}=\sum_{m}h_{mj}R_{mikl}+\sum_{m}h_{im}R_{mjkl}.
\end{equation}
We may use (2.2) without any mention. In \cite{9}, Colding-Minicozzi introduced the linear operator
\begin{equation}
  \mathcal{L}=\Delta-\langle X,\nabla(\cdot)\rangle
\end{equation}
on self-shrinkers, where $\Delta$ and $\nabla$ denote the Laplacian and the gradient operator, respectively.
 By direct calculation, we have the following basic formulas:
\begin{gather}
  \nabla_{i}H=\sum_{k}h_{ik}\langle X,e_{k}\rangle ,\\
  \nabla_{j}\nabla_{i}H=\sum_{k}h_{ijk}\langle X,e_{k}\rangle +h_{ij}-H\sum_{k}h_{ik}h_{kj}.
\end{gather}
By a similar derivation for Simons' identity, we have the self-shrinker version:
\begin{equation}
  \frac{1}{2}\mathcal{L}S=\sum_{i,j,k}h_{ijk}^{2}+S(1-S).
\end{equation}
Moreover, we have
\begin{equation}
  \frac{1}{3}\mathcal{L}f_{3}=2\sum_{i,j,k,l}h_{ik}h_{ijl}h_{kjl}+f_{3}(1-S),
\end{equation}
and
\begin{equation}
  \frac{1}{4}\mathcal{L}f_{4}=2A+B+f_{4}(1-S),
\end{equation}
where $A=\sum_{i,j,k,l,m}h_{km}h_{lm}h_{ijk}h_{ijl}$ and $B=\sum_{i,j,k,l,m}h_{il}h_{jm}h_{ijk}h_{klm}$. Regarding $A$ and $B$ in (2.9), through a discussion similar to \cite{15}, we have the following conclusion:
\begin{lemma}
 Let $X:M^{n}\rightarrow \mathbb{R}^{n+1}$ be a complete self-shrinker with constant $S$ and $f_{3}$. Then $A-2B=Sf_{4}-f_{3}^{2}$. Moreover, if $f_{4}$ is also constant, then we have $A=\frac{1}{5}(3Sf_{4}-2f_{4}-f_{3}^{2})$ and $B=\frac{1}{5}(2f_{3}^{2}-Sf_{4}-f_{4})$.
\end{lemma}
\begin{proof}
We will compute at an arbitrarily chosen point $p\in M$. Choose local orthonormal frame fields $\{e_{i}\}_{i=1}^{n}$ around $p$ such that $h_{ij}=\lambda_{i}\delta_{ij}$ at $p$. Since $S=$constant and $f_{3}=$constant, at $p$ we have
\begin{align*}
0&=\frac{1}{3}\sum_{i,j}h_{ij}\nabla_{j}\nabla_{i}f_{3}\\
&=\frac{1}{3}\sum_{k}\lambda_{k}\nabla_{k}\nabla_{k}f_{3}\\
&=\sum_{k}\lambda_{k}\left(\sum_{i}h_{iikk}\lambda_{i}^{2}+2\sum_{i,j}h_{ijk}^{2}\lambda_{i}\right)\\
&=\sum_{i,k}h_{iikk}\lambda_{k}\lambda_{i}^{2}+2\sum_{i,j,k}h_{ijk}^{2}\lambda_{i}\lambda_{k}\\
&=\sum_{i,k}\left(h_{kkii}+\lambda_{i}^{2}\lambda_{k}\delta_{ik}-\lambda_{i}\lambda_{k}^{2}+\lambda_{k}\lambda_{i}^{2}-\lambda_{i}\lambda_{k}^{2}\delta_{ik}\right)\lambda_{k}\lambda_{i}^{2}+2B\\
&=\sum_{i,k}h_{kkii}\lambda_{k}\lambda_{i}^{2}+\sum_{i,k}\lambda_{i}^{4}\lambda_{k}^{2}\delta_{ik}-\sum_{i,k}\lambda_{i}^{3}\lambda_{k}^{3}+\sum_{i,k}\lambda_{i}^{4}\lambda_{k}^{2}
-\sum_{i,k}\lambda_{i}^{3}\lambda_{k}^{3}\delta_{ik}+2B\\
&=\sum_{i}\left(\frac{1}{2}\nabla_{i}\nabla_{i}S-\sum_{j,k}h_{ijk}^{2}\right)\lambda_{i}^{2}+f_{6}-f_{3}^{2}+Sf_{4}-f_{6}+2B\\
&=Sf_{4}-f_{3}^{2}-(A-2B).
\end{align*}
In the fifth equality above we used the Ricci identities (2.3) and the Gauss equations (2.1). In the seventh equality above we used the equality $\frac{1}{2}\nabla_{i}\nabla_{i}S=\sum_{j,k}h_{ijk}^{2}+\sum_{k}h_{kkii}\lambda_{k}$ since $S=\sum_{j,k}h_{jk}^{2}$. Then we have $A-2B=Sf_{4}-f_{3}^{2}$ on $M$. If $f_{4}$ is also constant, (2.9) gives $2A+B=f_{4}(S-1)$. Thus $A=\frac{1}{5}(3Sf_{4}-2f_{4}-f_{3}^{2})$ and $B=\frac{1}{5}(2f_{3}^{2}-Sf_{4}-f_{4})$. The lemma follows.
\end{proof}
The following proposition reveals the benefit of $f_{3}=0$.
\begin{proposition}
Given four real numbers $a,b,c$ and $d$ $(a\geq b\geq c\geq d)$. If $a+b+c+d=0$ and $a^{3}+b^{3}+c^{3}+d^{3}=0$, then $a+d=b+c=0$.
\end{proposition}
\begin{proof}
It's easy to see that $a\geq0\geq d$. We obtain $-ad\geq 0$ with equality if and only if $a=b=c=d=0$. By direct calculation, we have
\begin{equation*}
\begin{split}
0&=a^{3}+b^{3}+c^{3}+d^{3}\\
&=(a+d)(a^{2}-ad+d^{2})+(b+c)(b^{2}-bc+c^{2})\\
&=(a+d)(a^{2}-ad+d^{2})-(a+d)(b^{2}-bc+c^{2})\\
&=(a+d)(a^{2}-ad+d^{2}-b^{2}+bc-c^{2})\\
&=(a+d)\left[(a+d)^{2}-3ad-(b+c)^{2}+3bc \right]\\
&=-3(a+d)(ad-bc).
\end{split}
\end{equation*}
Thus $a+d=0$ or $ad-bc=0$. If $a+d=0$, we immediately have $b+c=0$ by $a+b+c+d=0$. Next, we can assume that $ad=bc<0$ by the previous argument. Since $a\geq b\geq c\geq d$, we conclude that $a\geq b>0>c\geq d$. Combining $ad=bc$, this yields $a=b$ and $c=d$. Thus $a+d=b+c=0$. The proposition holds.
\end{proof}

In order to prove our results, we need the following generalized maximum principle for self-shrinkers which is proved by Cheng and Peng \cite{7}.
\begin{lemma}
 Let $X:M^{n}\rightarrow \mathbb{R}^{n+p}$ be a complete self-shrinker with Ricci curvature bounded from below. Let $f$ be any $C^{2}$-function bounded from above on this self-shrinker. Then, there exists a sequence of points $\{p_{t}\}\in M^{n}$, such that
\begin{equation*}
  \lim_{t\rightarrow\infty}f(X(p_{t}))=\sup f,\ \ \lim_{t\rightarrow\infty}|\nabla f|(X(p_{t}))=0,\ \ \limsup_{t\rightarrow\infty}\mathcal{L}f(X(p_{t}))\leq 0.
\end{equation*}
\end{lemma}
The following lemma proved by Cheng and Peng \cite{7} is also important for us:
\begin{lemma}
For an $n$-dimensional complete self-shrinker $X:M^{n}\rightarrow \mathbb{R}^{n+1}$ with $\inf H^{2}>0$, if the squared norm $S$ of the second fundamental form is constant, then $M^{n}$ is isometric to either $S^{n}(\sqrt{n})$ or $S^{m}(\sqrt{m})\times \mathbb{R}^{n-m}$ in $\mathbb{R}^{n+1}$, $1\leq m\leq n-1$.
\end{lemma}
The following proposition will be used in Section 3.
\begin{proposition}
Let $X:M^{4}\rightarrow \mathbb{R}^{5}$ be a $4$-dimensional complete self-shrinker with $S=4$, $f_{3}=0$ and $f_{4}=8$. Then $H=0$.
\end{proposition}
\begin{proof}
Since $S$ is constant, from the Gauss equations (2.1) and Cauchy-Schwarz inequality, we know that the Ricci curvature of $M$ is bounded from below and $H^{2}(\leq 4S)$ is bounded. Then we can apply the generalized maximum principle (Lemma 2.3) for $\mathcal{L}$-operator to $H$ and $-H$. Let's discuss $H$ first. Denote $\sup H$ by $\overline{H}$. There exists a sequence $\{p_{t}\}$ in $M$ such that
\begin{equation}
  \lim_{t\rightarrow\infty}H(X(p_{t}))=\overline{H},\ \ \lim_{t\rightarrow\infty}|\nabla H|(X(p_{t}))=0,\ \ \limsup_{t\rightarrow\infty}\mathcal{L}H(X(p_{t}))\leq 0.
\end{equation}
Because there may not be a global orthonormal frame field on $M$. At each point $p\in M$, we choose orthonormal basis $\{e_{i}(t)\}_{i=1}^{4}$ on $T_{p}M$ such that $h_{ij}(X(p))=\mu_{i}\delta_{ij}$ for $1\leq i,j\leq4$.
Since $S$ and $f_3$ are constant, combining with (2.7), we know that $\{h_{ij}(p_{t})\}$ and $\{h_{ijk}(p_{t})\}$ are bounded sequences for $i,j,k=1,2,3,4$. After passing to a subsequence (still denoted by $\{p_{t}\}$), we can assume that
\begin{equation*}
  \lim_{t\rightarrow\infty}H(X(p_{t}))=\overline{H},\ \ \lim_{t\rightarrow\infty}h_{ij}(X(p_{t}))=\overline{h}_{ij}=\overline{\mu}_{i}\delta_{ij},\ \ \lim_{t\rightarrow\infty}h_{ijk}(X(p_{t}))=\overline{h}_{ijk},
\end{equation*}
for $i,j,k=1,2,3,4$.
From (2.10), we have
\begin{align}
\overline{\mu}_{1}+\overline{\mu}_{2}+\overline{\mu}_{3}+\overline{\mu}_{4}&=\overline{H},\\
\overline{h}_{11i}+\overline{h}_{22i}+\overline{h}_{33i}+\overline{h}_{44i}&=0,\ \ i=1,2,3,4.
\end{align}
Since $S=4$, $f_{3}=0$ and $f_{4}=8$ on $M$, by sending $t\rightarrow\infty$, we have
\begin{align}
\overline{\mu}_{1}^{2}+\overline{\mu}_{2}^{2}+\overline{\mu}_{3}^{2}+\overline{\mu}_{4}^{2}&=4,\\
\overline{\mu}_{1}^{3}+\overline{\mu}_{2}^{3}+\overline{\mu}_{3}^{3}+\overline{\mu}_{4}^{3}&=0,\\
\overline{\mu}_{1}^{4}+\overline{\mu}_{2}^{4}+\overline{\mu}_{3}^{4}+\overline{\mu}_{4}^{4}&=8.
\end{align}
By differentiating $S$, $f_{3}$ and $f_{4}$ at $p_t$ once and sending $t\rightarrow\infty$, we get
\begin{align}
&\overline{\mu}_{1}\overline{h}_{11i}+\overline{\mu}_{2}\overline{h}_{22i}+\overline{\mu}_{3}\overline{h}_{33i}+\overline{\mu}_{4}\overline{h}_{44i}=0,\\
&\overline{\mu}_{1}^{2}\overline{h}_{11i}+\overline{\mu}_{2}^{2}\overline{h}_{22i}+\overline{\mu}_{3}^{2}\overline{h}_{33i}+\overline{\mu}_{4}^{2}\overline{h}_{44i}=0,\\
&\overline{\mu}_{1}^{3}\overline{h}_{11i}+\overline{\mu}_{2}^{3}\overline{h}_{22i}+\overline{\mu}_{3}^{3}\overline{h}_{33i}+\overline{\mu}_{4}^{3}\overline{h}_{44i}=0,\  \mathrm{for}\ i=1,2,3,4.
\end{align}
Evaluating (2.7) and (2.8) at $p_{t}$ and sending $t\rightarrow\infty$, we obtain
\begin{align}
\sum_{i,j,k}\overline{h}_{ijk}^{2}=12,\\
\sum_{i,j,k}\overline{\mu}_{i}\overline{h}_{ijk}^{2}=0.
\end{align}
Because of $S=4$, $f_{3}=0$ and $f_{4}=8$, by using of Lemma 2.1, we get $A=16$, that is
\begin{equation}
\sum_{i,j,k}\overline{\mu}_{i}^{2}\overline{h}_{ijk}^{2}=16.
\end{equation}
We have to consider four cases.

\textbf{Case 1. $\overline{\mu}_{1}$, $\overline{\mu}_{2}$, $\overline{\mu}_{3}$ and $\overline{\mu}_{4}$ are all equal.}\\
Since $\overline{\mu}_{1}=\overline{\mu}_{2}=\overline{\mu}_{3}=\overline{\mu}_{4}$, by (2.14), we get $\overline{\mu}_{1}=\overline{\mu}_{2}=\overline{\mu}_{3}=\overline{\mu}_{4}=0$, that is $\overline{H}=0$.

\textbf{Case 2. Three of $\overline{\mu}_{1}$, $\overline{\mu}_{2}$, $\overline{\mu}_{3}$ and $\overline{\mu}_{4}$ are equal exactly.}\\
Without loss of generality, we assume that $\overline{\mu}_{1}\neq\overline{\mu}_{2}=\overline{\mu}_{3}=\overline{\mu}_{4}$. By (2.14), we get $\overline{\mu}_{1}^{3}+3\overline{\mu}_{2}^{3}=0$, that is $\overline{\mu}_{1}=-3^{\frac{1}{3}}\overline{\mu}_{2}$. Substituting it into (2.13) and (2.15), we have $(3^{\frac{2}{3}}+3)\overline{\mu}_{2}^{2}=4$ and $(3^{\frac{4}{3}}+3)\overline{\mu}_{2}^{4}=8$. After cancelling $\overline{\mu}_{2}$, we get $(3^{\frac{2}{3}}+3)^{2}=2(3^{\frac{4}{3}}+3)$. It's a contradiction.

\textbf{Case 3. Two of $\overline{\mu}_{1}$, $\overline{\mu}_{2}$, $\overline{\mu}_{3}$ and $\overline{\mu}_{4}$ are equal exactly.}\\
Without loss of generality, we assume that $\overline{\mu}_{3}=\overline{\mu}_{4}$.\\
\textbf{Subcase 1. $\overline{\mu}_{1}=\overline{\mu}_{2}\neq\overline{\mu}_{3}$.}\\
By (2.14), we get $2\overline{\mu}_{1}^{3}+2\overline{\mu}_{3}^{3}=0$, that is $\overline{\mu}_{1}=-\overline{\mu}_{3}$. Substituting it into (2.13) and (2.15), we have $\overline{\mu}_{1}^{2}=1$ and $\overline{\mu}_{1}^{4}=2$ respectively. It's a contradiction.\\
\textbf{Subcase 2. $\overline{\mu}_{1}$, $\overline{\mu}_{2}$ and $\overline{\mu}_{3}$ are not equal to each other.}\\
Since $\overline{\mu}_{3}=\overline{\mu}_{4}$, from (2.12), (2.16) and (2.17), we have
\begin{align*}
\overline{h}_{11i}+\overline{h}_{22i}+(\overline{h}_{33i}+\overline{h}_{44i})&=0,\\
\overline{\mu}_{1}\overline{h}_{11i}+\overline{\mu}_{2}\overline{h}_{22i}+\overline{\mu}_{3}(\overline{h}_{33i}+\overline{h}_{44i})&=0,\\
\overline{\mu}_{1}^{2}\overline{h}_{11i}+\overline{\mu}_{2}^{2}\overline{h}_{22i}+\overline{\mu}_{3}^{2}(\overline{h}_{33i}+\overline{h}_{44i})&=0, \ \mathrm{for}\ i=1,2,3,4.
\end{align*}
Thus
\begin{equation}
\overline{h}_{11i}=\overline{h}_{22i}=0,\ \ \overline{h}_{33i}+\overline{h}_{44i}=0, \ \mathrm{for}\ i=1,2,3,4.
\end{equation}
Then
\begin{equation*}
\begin{aligned}
\sum_{i,j,k}\overline{h}_{ijk}^{2}&=
\begin{aligned}[t]
&6(\overline{h}_{123}^{2}+\overline{h}_{124}^{2}+\overline{h}_{134}^{2}+\overline{h}_{234}^{2})+3(\overline{h}_{331}^{2}+\overline{h}_{332}^{2}+\overline{h}_{333}^{2}+\overline{h}_{334}^{2}\\
&+\overline{h}_{441}^{2}+\overline{h}_{442}^{2}+\overline{h}_{443}^{2}+\overline{h}_{444}^{2})\\
\end{aligned}\\
&=6(\overline{h}_{123}^{2}+\overline{h}_{124}^{2}+\overline{h}_{134}^{2}+\overline{h}_{234}^{2}+\overline{h}_{331}^{2}+\overline{h}_{332}^{2}+\overline{h}_{333}^{2}+\overline{h}_{334}^{2}).
\end{aligned}
\end{equation*}
Combining the above equality with (2.19), we obtain
\begin{equation}
\overline{h}_{123}^{2}+\overline{h}_{124}^{2}+\overline{h}_{134}^{2}+\overline{h}_{234}^{2}+\overline{h}_{331}^{2}+\overline{h}_{332}^{2}+\overline{h}_{333}^{2}+\overline{h}_{334}^{2}=2.
\end{equation}
By using of (2.22) and (2.23), (2.21) gives
\begin{align*}
16=&\sum_{i,j,k}\overline{\mu}_{i}^{2}\overline{h}_{ijk}^{2}\\
=&\sum_{j,k}(\overline{\mu}_{1}^{2}\overline{h}_{1jk}^{2}+\overline{\mu}_{2}^{2}\overline{h}_{2jk}^{2}+\overline{\mu}_{3}^{2}\overline{h}_{3jk}^{2}+\overline{\mu}_{4}^{2}\overline{h}_{4jk}^{2})\\
=&2\overline{\mu}_{1}^{2}(\overline{h}_{123}^{2}+\overline{h}_{124}^{2}+\overline{h}_{134}^{2})+2\overline{\mu}_{2}^{2}(\overline{h}_{123}^{2}+\overline{h}_{124}^{2}+\overline{h}_{234}^{2})\\
&+2\overline{\mu}_{3}^{2}(\overline{h}_{123}^{2}+\overline{h}_{134}^{2}+\overline{h}_{234}^{2})+2\overline{\mu}_{4}^{2}(\overline{h}_{124}^{2}+\overline{h}_{134}^{2}+\overline{h}_{234}^{2})\\
&+\overline{\mu}_{1}^{2}(\overline{h}_{331}^{2}+\overline{h}_{441}^{2})+\overline{\mu}_{2}^{2}(\overline{h}_{332}^{2}+\overline{h}_{442}^{2})+\overline{\mu}_{3}^{2}(\overline{h}_{333}^{2}+\overline{h}_{443}^{2})+\overline{\mu}_{4}^{2}(\overline{h}_{334}^{2}+\overline{h}_{444}^{2})\\
=&2\left[(\overline{\mu}_{1}^{2}+\overline{\mu}_{2}^{2}+\overline{\mu}_{3}^{2})\overline{h}_{123}^{2}
+(\overline{\mu}_{1}^{2}+\overline{\mu}_{2}^{2}+\overline{\mu}_{3}^{2})\overline{h}_{124}^{2}+(\overline{\mu}_{1}^{2}+\overline{\mu}_{2}^{2}+\overline{\mu}_{3}^{2})\overline{h}_{134}^{2} \right.\\
&+\left.(\overline{\mu}_{2}^{2}+\overline{\mu}_{3}^{2}+\overline{\mu}_{4}^{2})\overline{h}_{234}^{2}+\overline{\mu}_{1}^{2}\overline{h}_{331}^{2}
+\overline{\mu}_{2}^{2}\overline{h}_{332}^{2}+\overline{\mu}_{3}^{2}\overline{h}_{333}^{2}+\overline{\mu}_{4}^{2}\overline{h}_{334}^{2}  \right]\\
\leq&2\left[(\overline{\mu}_{1}^{2}+\overline{\mu}_{2}^{2}+\overline{\mu}_{3}^{2}+\overline{\mu}_{4}^{2})\overline{h}_{123}^{2}
+(\overline{\mu}_{1}^{2}+\overline{\mu}_{2}^{2}+\overline{\mu}_{3}^{2}+\overline{\mu}_{4}^{2})\overline{h}_{124}^{2}
+(\overline{\mu}_{1}^{2}+\overline{\mu}_{2}^{2}+\overline{\mu}_{3}^{2}+\overline{\mu}_{4}^{2})\overline{h}_{134}^{2} \right.\\
&+(\overline{\mu}_{1}^{2}+\overline{\mu}_{2}^{2}+\overline{\mu}_{3}^{2}+\overline{\mu}_{4}^{2})\overline{h}_{234}^{2}
+(\overline{\mu}_{1}^{2}+\overline{\mu}_{2}^{2}+\overline{\mu}_{3}^{2}+\overline{\mu}_{4}^{2})\overline{h}_{331}^{2}
+(\overline{\mu}_{1}^{2}+\overline{\mu}_{2}^{2}+\overline{\mu}_{3}^{2}+\overline{\mu}_{4}^{2})\overline{h}_{332}^{2}\\
&\left.+(\overline{\mu}_{1}^{2}+\overline{\mu}_{2}^{2}+\overline{\mu}_{3}^{2}+\overline{\mu}_{4}^{2})\overline{h}_{333}^{2}
+(\overline{\mu}_{1}^{2}+\overline{\mu}_{2}^{2}+\overline{\mu}_{3}^{2}+\overline{\mu}_{4}^{2})\overline{h}_{334}^{2}     \right]\\
=&2S(\overline{h}_{123}^{2}+\overline{h}_{124}^{2}+\overline{h}_{134}^{2}+\overline{h}_{234}^{2}+\overline{h}_{331}^{2}+\overline{h}_{332}^{2}+\overline{h}_{333}^{2}+\overline{h}_{334}^{2})\\
=&16.
\end{align*}
Thus the inequality above is actually an equality. This gives
\begin{equation}
\overline{\mu}_{4}^{2}\overline{h}_{123}^{2}=\overline{\mu}_{3}^{2}\overline{h}_{124}^{2}=\overline{\mu}_{2}^{2}\overline{h}_{134}^{2}=\overline{\mu}_{1}^{2}\overline{h}_{234}^{2}=0
\end{equation}
and
\begin{equation}
(\overline{\mu}_{2}^{2}+\overline{\mu}_{3}^{2}+\overline{\mu}_{4}^{2})\overline{h}_{331}^{2}
=(\overline{\mu}_{1}^{2}+\overline{\mu}_{3}^{2}+\overline{\mu}_{4}^{2})\overline{h}_{332}^{2}
=(\overline{\mu}_{1}^{2}+\overline{\mu}_{2}^{2}+\overline{\mu}_{4}^{2})\overline{h}_{333}^{2}
=(\overline{\mu}_{1}^{2}+\overline{\mu}_{2}^{2}+\overline{\mu}_{3}^{2})\overline{h}_{334}^{2}=0.
\end{equation}
Since $\overline{\mu}_{1}$, $\overline{\mu}_{2}$ and $\overline{\mu}_{3}$ are not equal to each other, from (2.25) we have
\begin{equation}
\overline{h}_{331}=\overline{h}_{332}=\overline{h}_{333}=\overline{h}_{334}=0.
\end{equation}
If $\overline{\mu}_{3}=\overline{\mu}_{4}=0$, then $\overline{\mu}_{1}$ and $\overline{\mu}_{2}$ are not equal to $0$. From (2.24) we have
\begin{equation}
\overline{h}_{134}=\overline{h}_{234}=0.
\end{equation}
Combining (2.23), (2.26) and (2.27), we get
\begin{equation}
\overline{h}_{123}^{2}+\overline{h}_{124}^{2}=2.
\end{equation}
Then (2.20) gives
\begin{equation*}
(\overline{\mu}_{1}+\overline{\mu}_{2})\overline{h}_{123}^{2}+(\overline{\mu}_{1}+\overline{\mu}_{2})\overline{h}_{124}^{2}=0.
\end{equation*}
Hence $\overline{\mu}_{1}+\overline{\mu}_{2}=0$, that is $\overline{H}=0$. If $\overline{\mu}_{3}=\overline{\mu}_{4}\neq 0$, then (2.24) gives
\begin{equation}
\overline{h}_{123}=\overline{h}_{124}=0.
\end{equation}
Combining (2.23), (2.26) and (2.29), we get
\begin{equation}
\overline{h}_{134}^{2}+\overline{h}_{234}^{2}=2.
\end{equation}
Since $\overline{\mu}_{1}\neq\overline{\mu}_{2}$, from (2.24) and (2.30), we know that one of them is equal to $0$ and the other one is not. Without loss of generality, we assume that $\overline{\mu}_{1}=0$ and $\overline{\mu}_{2}\neq 0$. Then (2.24) and (2.30) give $\overline{h}_{134}=0$ and $\overline{h}_{234}^{2}=2$. Immediately, (2.20) gives $(\overline{\mu}_{2}+\overline{\mu}_{3}+\overline{\mu}_{4})\overline{h}_{234}^{2}=0$. Thus $\overline{\mu}_{2}+\overline{\mu}_{3}+\overline{\mu}_{4}=0$, that is $\overline{H}=0$.

\textbf{Case 4. $\overline{\mu}_{1}$, $\overline{\mu}_{2}$, $\overline{\mu}_{3}$ and $\overline{\mu}_{4}$ are not equal to each other.}\\
By using of (2.12), (2.16), (2.17) and (2.18), we have that
\begin{equation}
\overline{h}_{11i}=\overline{h}_{22i}=\overline{h}_{33i}=\overline{h}_{44i}=0, \ \mathrm{for}\ i=1,2,3,4.
\end{equation}
From (2.19) we get
\begin{equation}
\overline{h}_{123}^{2}+\overline{h}_{124}^{2}+\overline{h}_{134}^{2}+\overline{h}_{234}^{2}=2.
\end{equation}
Similar to the previous case, (2.21) gives
\begin{align*}
16=&\sum_{i,j,k}\overline{\mu}_{i}^{2}\overline{h}_{ijk}^{2}\\
=&2\left[(\overline{\mu}_{1}^{2}+\overline{\mu}_{2}^{2}+\overline{\mu}_{3}^{2})\overline{h}_{123}^{2}
+(\overline{\mu}_{1}^{2}+\overline{\mu}_{2}^{2}+\overline{\mu}_{3}^{2})\overline{h}_{124}^{2}+(\overline{\mu}_{1}^{2}+\overline{\mu}_{2}^{2}+\overline{\mu}_{3}^{2})\overline{h}_{134}^{2} \right.\\
&+\left.(\overline{\mu}_{2}^{2}+\overline{\mu}_{3}^{2}+\overline{\mu}_{4}^{2})\overline{h}_{234}^{2}\right]\\
\leq&2S(\overline{h}_{123}^{2}+\overline{h}_{124}^{2}+\overline{h}_{134}^{2}+\overline{h}_{234}^{2})\\
=&16.
\end{align*}
Thus the inequality above is actually an equality. This gives
\begin{equation}
\overline{\mu}_{4}^{2}\overline{h}_{123}^{2}=\overline{\mu}_{3}^{2}\overline{h}_{124}^{2}=\overline{\mu}_{2}^{2}\overline{h}_{134}^{2}=\overline{\mu}_{1}^{2}\overline{h}_{234}^{2}=0
\end{equation}
If all of $\overline{\mu}_{1}$, $\overline{\mu}_{2}$, $\overline{\mu}_{3}$ and $\overline{\mu}_{4}$ are not equal to $0$, from (2.33) we know that $\overline{h}_{123}=\overline{h}_{124}=\overline{h}_{134}=\overline{h}_{234}=0$. It's a contradiction. Since $\overline{\mu}_{1}$, $\overline{\mu}_{2}$, $\overline{\mu}_{3}$ and $\overline{\mu}_{4}$ are not equal to each other, without loss of generality, we can assume that $\overline{\mu}_{4}=0$. Then we have
\begin{equation*}
\overline{h}_{124}=\overline{h}_{134}=\overline{h}_{234}=0,\ \ \overline{h}_{123}^{2}=2.
\end{equation*}
Combining the above equalities, from (2.20) we get $(\overline{\mu}_{1}+\overline{\mu}_{2}+\overline{\mu}_{3})\overline{h}_{123}^{2}=0$. Hence $\overline{\mu}_{1}+\overline{\mu}_{2}+\overline{\mu}_{3}=0$, that is $\overline{H}=0$.

In conclusion, $\sup H=\overline{H}=0$. By applying the generalized maximum principle for $\mathcal{L}$-operator to $-H$ and using the similar argument, we obtain that $\inf H=0$. Hence $H=0$. The lemma follows.

\end{proof}

\section{The case of $f_{3}=0$ when $H$ vanishes somewhere}
From now on, we suppose that $X:M^{4}\rightarrow \mathbb{R}^{5}$ be a $4$-dimensional complete self-shrinker in $\mathbb{R}^{5}$ with constant $S$, $f_{3}=0$ and constant $f_{4}$. From (2.7), we know that $S=0$ or $S\geq 1$. If $S=0$, then $M$ is isometric to $\mathbb{R}^{4}$. Next, we assume that $S\geq 1$. By Lemma 2.4, it's easy to see that $\inf |H|=0$ since $f_{3}=0$. In this section, we will deal with the case when $H$ vanishes somewhere. Assume that $H(P)=0, P\in M$. All the rest of the calculations will be done at $P$ in this section. We take orthonormal frames $\{e_{i}\}_{i=1}^{4}$ such that $h_{ij}=\lambda_{i}\delta_{ij}$ at $P$, for $1\leq i,j\leq 4$, and $\lambda_{1}\geq\lambda_{2}\geq\lambda_{3}\geq\lambda_{4}$. Note that $H(P)=0$ and $f_{3}=0$, by Proposition 2.2, we obtain
\begin{equation}
\lambda_{1}+\lambda_{4}=0,\ \ \ \lambda_{2}+\lambda_{3}=0.
\end{equation}
We will begin with the first and the second derivatives of $H$, $S$, $f_{3}$ and $f_{4}$. By (2.5) and (2.6), we have
\begin{equation}
\nabla_{i}H=\lambda_{i}\langle X(P),e_{i}\rangle,\ \ \ i=1,2,3,4,
\end{equation}
and
\begin{equation}
\left\{\begin{aligned}
&\nabla_{i}\nabla_{i}H=\sum\limits_{k}h_{iik}\langle X(P),e_{k}\rangle+\lambda_{i},\ \ \ i=1,2,3,4,\\
&\nabla_{j}\nabla_{i}H=\sum\limits_{k}h_{ijk}\langle X(P),e_{k}\rangle,\ \ \ i\neq j,\ \ \ i,j=1,2,3,4.
\end{aligned}\right.
\end{equation}

Since $S$ is constant, after taking the first and the second derivatives at $P$, we have
\begin{equation*}
\sum\limits_{i,j}h_{ij}h_{ijk}=0,\ \ \ k=1,2,3,4,
\end{equation*}
\begin{equation*}
\sum\limits_{i,j}h_{ij}h_{ijkl}+\sum\limits_{i,j}h_{ijk}h_{ijl}=0,\ \ \ k,l=1,2,3,4.
\end{equation*}
that is
\begin{equation}
\sum_{i}\lambda_{i}h_{iik}=0,\ \ \ k=1,2,3,4.
\end{equation}
and
\begin{equation}
\left\{\begin{aligned}
&\sum\limits_{i}\lambda_{i}h_{iikk}=-\sum\limits_{i,j}h_{ijk}^{2},\ \ \ k=1,2,3,4,\\
&\sum\limits_{i}\lambda_{i}h_{iikl}=-\sum\limits_{i,j}h_{ijk}h_{ijl},\ \ \ k\neq l,\ \ \ k,l=1,2,3,4.
\end{aligned}\right.
\end{equation}
Similarly, because $f_{3}=0$, we have
\begin{equation*}
\sum\limits_{i,j,k}h_{ij}h_{jk}h_{kil}=0,\ \ \ l=1,2,3,4,
\end{equation*}
\begin{equation*}
2\sum\limits_{i,j,k}h_{ij}h_{kil}h_{kjm}+\sum\limits_{i,j,k}h_{ij}h_{jk}h_{kilm}=0,\ \ \ l,m=1,2,3,4,
\end{equation*}
that is
\begin{equation}
\sum\limits_{i}\lambda_{i}^{2}h_{iik}=0,\ \ \ k=1,2,3,4,
\end{equation}
and
\begin{equation}
\left\{\begin{aligned}
&\sum\limits_{i}\lambda_{i}^{2}h_{iikk}=-2\sum\limits_{i,j}\lambda_{i}h_{ijk}^{2},\ \ \ k=1,2,3,4,\\
&\sum\limits_{i}\lambda_{i}^{2}h_{iikl}=-2\sum\limits_{i,j}\lambda_{i}h_{ijk}h_{ijl},\ \ \ k\neq l,\ \ \ k,l=1,2,3,4.
\end{aligned}\right.
\end{equation}
Since $f_{4}=$constant, we obtain
\begin{equation*}
\sum\limits_{i,j,k,l}h_{ij}h_{jk}h_{kl}h_{lim}=0,\ \ \ m=1,2,3,4,
\end{equation*}
\begin{equation*}
2\sum\limits_{i,j,k,l}h_{ij}h_{jk}h_{klr}h_{lim}+\sum\limits_{i,j,k,l}h_{ij}h_{kl}h_{jkr}h_{lim}+\sum\limits_{i,j,k,l}h_{ij}h_{jk}h_{kl}h_{limr}=0,\ \ \ m,r=1,2,3,4,
\end{equation*}
that is
\begin{equation}
\sum\limits_{i}\lambda_{i}^{3}h_{iik}=0,\ \ \ k=1,2,3,4,
\end{equation}
and
\begin{equation}
\left\{\begin{aligned}
&\sum\limits_{i}\lambda_{i}^{3}h_{iikk}=-\sum\limits_{i,j}(2\lambda_{i}^{2}+\lambda_{i}\lambda_{j})h_{ijk}^{2},\ \ \ k=1,2,3,4.\\
&\sum\limits_{i}\lambda_{i}^{3}h_{iikl}=-\sum\limits_{i,j}(2\lambda_{i}^{2}+\lambda_{i}\lambda_{j})h_{ijk}h_{ijl},\ \ \ k\neq l,\ \ \ k,l=1,2,3,4.
\end{aligned}\right.
\end{equation}
By (2.1) and (2.3), we have the following commutation formulas
\begin{equation*}
\begin{aligned}
&h_{ijkl}-h_{ijlk}=
\begin{aligned}[t]
&\lambda_{j}\lambda_{k}\lambda_{i}\delta_{jk}\delta_{il}-\lambda_{j}\lambda_{l}\lambda_{i}\delta_{jl}\delta_{ik}\\
&+\lambda_{i}\lambda_{k}\lambda_{j}\delta_{ik}\delta_{jl}-\lambda_{i}\lambda_{l}\lambda_{j}\delta_{il}\delta_{jk},\\
\end{aligned}\\
\end{aligned}
\end{equation*}
that is
\begin{equation}
\left\{\begin{aligned}
&h_{ijkl}=h_{ijlk},\ \ \ \mathrm{if}\  i,j\  \mathrm{and}\ k\ \mathrm{are\ not\ equal\ to\ each\ other},\\
&h_{iikl}=h_{iilk},\ \ \ \mathrm{for}\ 1\leq i,k,l\leq 4,\\
&h_{1212}-h_{1221}=\lambda_{1}\lambda_{2}(\lambda_{1}-\lambda_{2}),\ \ h_{1313}-h_{1331}=\lambda_{1}\lambda_{3}(\lambda_{1}-\lambda_{3}),\\
&h_{1414}-h_{1441}=\lambda_{1}\lambda_{4}(\lambda_{1}-\lambda_{4}),\ \ h_{2323}-h_{2332}=\lambda_{2}\lambda_{3}(\lambda_{2}-\lambda_{3}),\\
&h_{2424}-h_{2442}=\lambda_{2}\lambda_{4}(\lambda_{2}-\lambda_{4}),\ \ h_{3434}-h_{3443}=\lambda_{3}\lambda_{4}(\lambda_{3}-\lambda_{4}),\\
\end{aligned}\right.
\end{equation}
According to (3.1), we have to consider three cases of principal curvature at $P$.

\textbf{Case \rmnum{1}: $\lambda_{1}=\lambda_{2}>0$.}\\
By (3.1), we have
\begin{equation}
\lambda_{1}=\lambda_{2}=-\lambda_{3}=-\lambda_{4}>0.
\end{equation}
From (3.4) and (3.6), we obtain
\begin{equation*}
\left\{\begin{aligned}
&\lambda_{1}h_{11k}+\lambda_{1}h_{22k}-\lambda_{1}h_{33k}-\lambda_{1}h_{44k}=0,\\
&\lambda_{1}^{2}h_{11k}+\lambda_{1}^{2}h_{22k}+\lambda_{1}^{2}h_{33k}+\lambda_{1}^{2}h_{44k}=0,\ \ \ k=1,2,3,4.
\end{aligned}\right.
\end{equation*}
Thus,
\begin{equation}
h_{11k}=-h_{22k},\ \ \ h_{33k}=-h_{44k},\ \ \ k=1,2,3,4.
\end{equation}
By (3.2), (3.11) and (3.12), we have
\begin{equation}
\langle X(P),e_{i}\rangle=0,\ \ \ i=1,2,3,4.
\end{equation}
Taking (3.13) into (3.3), we obtain
\begin{equation}
\nabla_{i}\nabla_{i}H=\lambda_{i},\ \ \ i=1,2,3,4.
\end{equation}
Multiplying both sides of the first set of equalities of (3.5) by $\lambda_{1}^{2}$, and combining (3.11), we conclude
\begin{equation*}
\sum\limits_{i}\lambda_{i}^{3}h_{iikk}=-\lambda_{1}^{2}\sum\limits_{i,j}h_{ijk}^{2},\ \ \ k=1,2,3,4.
\end{equation*}
Taking the above equalities into the first set of equalities of (3.9), and combining (3.11), we have
\begin{equation}
\sum\limits_{i,j}(\lambda_{1}^{2}+\lambda_{i}\lambda_{j})h_{ijk}^{2}=0,\ \ \ k=1,2,3,4.
\end{equation}
Pay attention to (3.11). Since $\lambda_{1}^{2}+\lambda_{i}\lambda_{j}\geq 0$ and $h_{ijk}^{2}\geq 0$ for any $1\leq i,j,k\leq 4$, we have
\begin{equation}
(\lambda_{1}^{2}+\lambda_{i}\lambda_{j})h_{ijk}^{2}=0,\ \ \ \mathrm{for}\ \  1\leq i,j,k\leq 4.
\end{equation}
Setting $i=j$ in (3.16), yields
\begin{equation}
h_{iik}=0,\ \ \ \mathrm{for}\ \ \ 1\leq i,k\leq 4.
\end{equation}
Setting $i=1, j=2$ in (3.16) yields
\begin{equation}
h_{123}=h_{124}=0.
\end{equation}
Setting $i=3, j=4$ in (3.16) yields
\begin{equation}
h_{341}=h_{342}=0.
\end{equation}
By (2.2), (3.17), (3.18) and (3.19), we have
\begin{equation*}
h_{ijk}=0,\ \ \ \mathrm{for}\ \ \ 1\leq i,j,k\leq 4.
\end{equation*}
This yields $S\equiv1$ and $\nabla h\equiv 0$ by (2.7). Thus $M$ is isometric to one of $S^{k}(\sqrt{k})\times\mathbb{R}^{4-k}(0\leq k \leq 3)$ and $S^{4}(\sqrt{4})$. It's a contradiction with $f_{3}=0$.

\textbf{Case \rmnum{2}: $\lambda_{1}, \lambda_{2}, \lambda_{3}$ and $\lambda_{4}$ are not equal to each other.}\\
In this case, by (3.1), we know that all principal curvatures don't vanish at $P$. By (3.1), (3.4), (3.6) and (3.8), we have
\begin{equation*}
\left\{\begin{aligned}
&\lambda_{1}h_{11k}+\lambda_{2}h_{22k}-\lambda_{2}h_{33k}=\lambda_{1}h_{44k},\\
&\lambda_{1}^{2}h_{11k}+\lambda_{1}^{2}h_{22k}+\lambda_{2}^{2}h_{33k}=-\lambda_{1}^{2}h_{44k},\\
&\lambda_{1}^{3}h_{11k}+\lambda_{2}^{3}h_{22k}-\lambda_{2}^{3}h_{33k}=\lambda_{1}^{3}h_{44k},\\
\end{aligned}\right.
\end{equation*}
for $k=1,2,3,4$. This can be considered as an linear equation system of $h_{11k}$, $h_{22k}$ and $h_{33k}$. By standard linear algebra theory, we obtain
\begin{equation}
h_{11k}=h_{44k},\ \ \ h_{22k}=h_{33k}=-\frac{\lambda_{1}^{2}}{\lambda_{2}^{2}}h_{44k},\ \ \ k=1,2,3,4.
\end{equation}
Taking (3.20) into (3.2), we have
\begin{equation*}
\langle X(P),e_{k}\rangle=\frac{2}{\lambda_{k}}\left(1-\frac{\lambda_{1}^{2}}{\lambda_{2}^{2}}\right)h_{44k},\ \ \ k=1,2,3,4,
\end{equation*}
that is
\begin{equation}
\left\{\begin{aligned}
&\langle X(P),e_{1}\rangle=\frac{2}{\lambda_{1}}\left(1-\frac{\lambda_{1}^{2}}{\lambda_{2}^{2}}\right)h_{441},\\
&\langle X(P),e_{2}\rangle=\frac{2}{\lambda_{2}}\left(1-\frac{\lambda_{1}^{2}}{\lambda_{2}^{2}}\right)h_{442},\\
&\langle X(P),e_{3}\rangle=-\frac{2}{\lambda_{2}}\left(1-\frac{\lambda_{1}^{2}}{\lambda_{2}^{2}}\right)h_{443},\\
&\langle X(P),e_{4}\rangle=-\frac{2}{\lambda_{1}}\left(1-\frac{\lambda_{1}^{2}}{\lambda_{2}^{2}}\right)h_{444}.
\end{aligned}\right.
\end{equation}
Taking (3.21) into (3.3), we obtain
\begin{equation}
\left\{\begin{aligned}
&H_{,11}=2\left(1-\frac{\lambda_{1}^{2}}{\lambda_{2}^{2}}\right)\left(\frac{1}{\lambda_{1}}h_{441}^{2}
+\frac{1}{\lambda_{2}}h_{442}^{2}-\frac{1}{\lambda_{2}}h_{443}^{2}-\frac{1}{\lambda_{1}}h_{444}^{2}\right)+\lambda_{1},\\
&H_{,22}=-\frac{2\lambda_{1}^{2}}{\lambda_{2}^{2}}\left(1-\frac{\lambda_{1}^{2}}{\lambda_{2}^{2}}\right)\left(\frac{1}{\lambda_{1}}h_{441}^{2}
+\frac{1}{\lambda_{2}}h_{442}^{2}-\frac{1}{\lambda_{2}}h_{443}^{2}-\frac{1}{\lambda_{1}}h_{444}^{2}\right)+\lambda_{2},\\
&H_{,33}=-\frac{2\lambda_{1}^{2}}{\lambda_{2}^{2}}\left(1-\frac{\lambda_{1}^{2}}{\lambda_{2}^{2}}\right)\left(\frac{1}{\lambda_{1}}h_{441}^{2}
+\frac{1}{\lambda_{2}}h_{442}^{2}-\frac{1}{\lambda_{2}}h_{443}^{2}-\frac{1}{\lambda_{1}}h_{444}^{2}\right)-\lambda_{2},\\
&H_{,44}=2\left(1-\frac{\lambda_{1}^{2}}{\lambda_{2}^{2}}\right)\left(\frac{1}{\lambda_{1}}h_{441}^{2}
+\frac{1}{\lambda_{2}}h_{442}^{2}-\frac{1}{\lambda_{2}}h_{443}^{2}-\frac{1}{\lambda_{1}}h_{444}^{2}\right)-\lambda_{1}.
\end{aligned}\right.
\end{equation}
By (3.1), (3.5) and (3.20), we obtain
\begin{equation}
\left\{
\begin{aligned}
&\lambda_{1}\left(h_{1111}-h_{4411}\right)+\lambda_{2}\left(h_{2211}-h_{3311}\right)=
\begin{aligned}[t]
&-2\left[\left(1+\frac{\lambda_{1}^{4}}{\lambda_{2}^{4}}\right)h_{441}^{2}+h_{442}^{2}+h_{443}^{2}\right.\\
&\left.+h_{444}^{2}+h_{123}^{2}+h_{124}^{2}+h_{134}^{2}\right],\\
\end{aligned}\\
&\lambda_{1}\left(h_{1122}-h_{4422}\right)+\lambda_{2}\left(h_{2222}-h_{3322}\right)=
\begin{aligned}[t]
&-2\left[\left(1+\frac{\lambda_{1}^{4}}{\lambda_{2}^{4}}\right)h_{442}^{2}
+\frac{\lambda_{1}^{4}}{\lambda_{2}^{4}}h_{441}^{2}+\frac{\lambda_{1}^{4}}{\lambda_{2}^{4}}h_{443}^{2}\right.\\
&\left.+\frac{\lambda_{1}^{4}}{\lambda_{2}^{4}}h_{444}^{2}
+h_{123}^{2}+h_{124}^{2}+h_{234}^{2}\right],\\
\end{aligned}\\
&\lambda_{1}\left(h_{1133}-h_{4433}\right)+\lambda_{2}\left(h_{2233}-h_{3333}\right)=
\begin{aligned}[t]
&-2\left[\left(1+\frac{\lambda_{1}^{4}}{\lambda_{2}^{4}}\right)h_{443}^{2}
+\frac{\lambda_{1}^{4}}{\lambda_{2}^{4}}h_{441}^{2}+\frac{\lambda_{1}^{4}}{\lambda_{2}^{4}}h_{442}^{2}\right.\\
&\left.+\frac{\lambda_{1}^{4}}{\lambda_{2}^{4}}h_{444}^{2}
+h_{123}^{2}+h_{134}^{2}+h_{234}^{2}\right],\\
\end{aligned}\\
&\lambda_{1}\left(h_{1144}-h_{4444}\right)+\lambda_{2}\left(h_{2244}-h_{3344}\right)=
\begin{aligned}[t]
&-2\left[\left(1+\frac{\lambda_{1}^{4}}{\lambda_{2}^{4}}\right)h_{444}^{2}+h_{441}^{2}+h_{442}^{2}\right.\\
&\left.+h_{443}^{2}+h_{124}^{2}+h_{134}^{2}+h_{234}^{2}\right].\\
\end{aligned}\\
\end{aligned}
\right.
\end{equation}
By (3.1), (3.7) and (3.20), we obtain
\begin{equation}
\left\{
\begin{aligned}
&\lambda_{1}^{2}\left(h_{1111}+h_{4411}\right)+\lambda_{2}^{2}\left(h_{2211}+h_{3311}\right)=
\begin{aligned}[t]
&-2\left[\left(\lambda_{1}+\lambda_{2}\right)h_{442}^{2}+\left(\lambda_{1}-\lambda_{2}\right)h_{443}^{2}\right.\\
&\left.-\left(\lambda_{1}-\lambda_{2}\right)h_{124}^{2}-\left(\lambda_{1}+\lambda_{2}\right)h_{134}^{2}\right],\\
\end{aligned}\\
&\lambda_{1}^{2}\left(h_{1122}+h_{4422}\right)+\lambda_{2}^{2}\left(h_{2222}+h_{3322}\right)=
\begin{aligned}[t]
&-2\left[\frac{\lambda_{1}^{4}}{\lambda_{2}^{4}}\left(\lambda_{1}+\lambda_{2}\right)h_{441}^{2}-\frac{\lambda_{1}^{4}}{\lambda_{2}^{4}}\left(\lambda_{1}-\lambda_{2}\right)h_{444}^{2}\right.\\
&\left.+\left(\lambda_{1}-\lambda_{2}\right)h_{123}^{2}-\left(\lambda_{1}+\lambda_{2}\right)h_{234}^{2}\right],\\
\end{aligned}\\
&\lambda_{1}^{2}\left(h_{1133}+h_{4433}\right)+\lambda_{2}^{2}\left(h_{2233}+h_{3333}\right)=
\begin{aligned}[t]
&-2\left[\frac{\lambda_{1}^{4}}{\lambda_{2}^{4}}\left(\lambda_{1}-\lambda_{2}\right)h_{441}^{2}-\frac{\lambda_{1}^{4}}{\lambda_{2}^{4}}\left(\lambda_{1}+\lambda_{2}\right)h_{444}^{2}\right.\\
&\left.+\left(\lambda_{1}+\lambda_{2}\right)h_{123}^{2}-\left(\lambda_{1}-\lambda_{2}\right)h_{234}^{2}\right],\\
\end{aligned}\\
&\lambda_{1}^{2}\left(h_{1144}+h_{4444}\right)+\lambda_{2}^{2}\left(h_{2244}+h_{3344}\right)=
\begin{aligned}[t]
&-2\left[-\left(\lambda_{1}-\lambda_{2}\right)h_{442}^{2}-\left(\lambda_{1}+\lambda_{2}\right)h_{443}^{2}\right.\\
&\left.+\left(\lambda_{1}+\lambda_{2}\right)h_{124}^{2}+\left(\lambda_{1}-\lambda_{2}\right)h_{134}^{2}\right].\\
\end{aligned}\\
\end{aligned}
\right.
\end{equation}
By (3.1), (3.9) and (3.20), we obtain
\begin{equation}
\left\{
\begin{aligned}
&\lambda_{1}^{3}\left(h_{1111}-h_{4411}\right)+\lambda_{2}^{3}\left(h_{2211}-h_{3311}\right)=
\begin{aligned}[t]
&-2\left[\frac{3\lambda_{1}^{2}\left(\lambda_{1}^{2}+\lambda_{2}^{2}\right)}{\lambda_{2}^{2}}h_{441}^{2}+\left(\lambda_{1}^{2}+\lambda_{2}^{2}+\lambda_{1}\lambda_{2}\right)h_{442}^{2}\right.\\
&+\left(\lambda_{1}^{2}+\lambda_{2}^{2}-\lambda_{1}\lambda_{2}\right)h_{443}^{2}+\lambda_{1}^{2}h_{444}^{2}+\lambda_{2}^{2}h_{123}^{2}\\
&\left.+\left(\lambda_{1}^{2}+\lambda_{2}^{2}-\lambda_{1}\lambda_{2}\right)h_{124}^{2}+\left(\lambda_{1}^{2}+\lambda_{2}^{2}+\lambda_{1}\lambda_{2}\right)h_{134}^{2}\right],\\
\end{aligned}\\
&\lambda_{1}^{3}\left(h_{1122}-h_{4422}\right)+\lambda_{2}^{3}\left(h_{2222}-h_{3322}\right)=
\begin{aligned}[t]
&-2\left[\frac{\lambda_{1}^{4}}{\lambda_{2}^{4}}\left(\lambda_{1}^{2}+\lambda_{2}^{2}+\lambda_{1}\lambda_{2}\right)h_{441}^{2}+\frac{3\lambda_{1}^{2}\left(\lambda_{1}^{2}+\lambda_{2}^{2}\right)}{\lambda_{2}^{2}}h_{442}^{2}\right.\\
&+\frac{\lambda_{1}^{4}}{\lambda_{2}^{2}}h_{443}^{2}+\frac{\lambda_{1}^{4}}{\lambda_{2}^{4}}\left(\lambda_{1}^{2}+\lambda_{2}^{2}-\lambda_{1}\lambda_{2}\right)h_{444}^{2}+\lambda_{1}^{2}h_{124}^{2}\\
&\left.+\left(\lambda_{1}^{2}+\lambda_{2}^{2}-\lambda_{1}\lambda_{2}\right)h_{123}^{2}+\left(\lambda_{1}^{2}+\lambda_{2}^{2}+\lambda_{1}\lambda_{2}\right)h_{234}^{2}\right],\\
\end{aligned}\\
&\lambda_{1}^{3}\left(h_{1133}-h_{4433}\right)+\lambda_{2}^{3}\left(h_{2233}-h_{3333}\right)=
\begin{aligned}[t]
&-2\left[\frac{\lambda_{1}^{4}}{\lambda_{2}^{4}}\left(\lambda_{1}^{2}+\lambda_{2}^{2}-\lambda_{1}\lambda_{2}\right)h_{441}^{2}+\frac{3\lambda_{1}^{2}\left(\lambda_{1}^{2}+\lambda_{2}^{2}\right)}{\lambda_{2}^{2}}h_{443}^{2}\right.\\
&+\frac{\lambda_{1}^{4}}{\lambda_{2}^{2}}h_{442}^{2}+\frac{\lambda_{1}^{4}}{\lambda_{2}^{4}}\left(\lambda_{1}^{2}+\lambda_{2}^{2}+\lambda_{1}\lambda_{2}\right)h_{444}^{2}+\lambda_{1}^{2}h_{134}^{2}\\
&\left.+\left(\lambda_{1}^{2}+\lambda_{2}^{2}+\lambda_{1}\lambda_{2}\right)h_{123}^{2}+\left(\lambda_{1}^{2}+\lambda_{2}^{2}-\lambda_{1}\lambda_{2}\right)h_{234}^{2}\right],\\
\end{aligned}\\
&\lambda_{1}^{3}\left(h_{1144}-h_{4444}\right)+\lambda_{2}^{3}\left(h_{2244}-h_{3344}\right)=
\begin{aligned}[t]
&-2\left[\frac{3\lambda_{1}^{2}\left(\lambda_{1}^{2}+\lambda_{2}^{2}\right)}{\lambda_{2}^{2}}h_{444}^{2}+\left(\lambda_{1}^{2}+\lambda_{2}^{2}-\lambda_{1}\lambda_{2}\right)h_{442}^{2}\right.\\
&+\left(\lambda_{1}^{2}+\lambda_{2}^{2}+\lambda_{1}\lambda_{2}\right)h_{443}^{2}+\lambda_{1}^{2}h_{441}^{2}+\lambda_{2}^{2}h_{234}^{2}\\
&\left.+\left(\lambda_{1}^{2}+\lambda_{2}^{2}+\lambda_{1}\lambda_{2}\right)h_{124}^{2}+\left(\lambda_{1}^{2}+\lambda_{2}^{2}-\lambda_{1}\lambda_{2}\right)h_{134}^{2}\right].\\
\end{aligned}\\
\end{aligned}
\right.
\end{equation}
Combining (3.22), (3.23), (3.24) and (3.25), we obtain
\begin{equation}
\begin{aligned}
&h_{4411}=
\begin{aligned}[t]
&\frac{2(\lambda_{1}^{4}+2\lambda_{1}^{2}\lambda_{2}^{2}-\lambda_{2}^{4})}{\lambda_{1}\lambda_{2}^{2}(\lambda_{1}^{2}-\lambda_{2}^{2})}h_{441}^{2}
+\frac{1}{\lambda_{2}}h_{442}^{2}-\frac{1}{\lambda_{2}}h_{443}^{2}+\frac{2}{\lambda_{1}+\lambda_{2}}h_{124}^{2}\\
&+\frac{2}{\lambda_{1}-\lambda_{2}}h_{134}^{2}-\frac{\lambda_{1}\lambda_{2}^{2}}{2(\lambda_{1}^{2}-\lambda_{2}^{2})},\\
\end{aligned}\\
\end{aligned}
\end{equation}

\begin{equation}
\begin{aligned}
&h_{3322}=
\begin{aligned}[t]
&\frac{2\lambda_{1}^{2}(\lambda_{2}^{4}+2\lambda_{1}^{2}\lambda_{2}^{2}-\lambda_{1}^{4})}{\lambda_{2}^{5}(\lambda_{2}^{2}-\lambda_{1}^{2})}h_{442}^{2}
+\frac{\lambda_{1}^{3}}{\lambda_{2}^{4}}h_{441}^{2}-\frac{\lambda_{1}^{3}}{\lambda_{2}^{4}}h_{444}^{2}+\frac{2}{\lambda_{1}+\lambda_{2}}h_{123}^{2}\\
&+\frac{2}{\lambda_{2}-\lambda_{1}}h_{234}^{2}+\frac{\lambda_{2}\lambda_{1}^{2}}{2(\lambda_{1}^{2}-\lambda_{2}^{2})},\\
\end{aligned}\\
\end{aligned}
\end{equation}

\begin{equation}
\begin{aligned}
&h_{2233}=
\begin{aligned}[t]
&\frac{2\lambda_{1}^{2}(\lambda_{2}^{4}+2\lambda_{1}^{2}\lambda_{2}^{2}-\lambda_{1}^{4})}{\lambda_{2}^{5}(\lambda_{1}^{2}-\lambda_{2}^{2})}h_{443}^{2}
+\frac{\lambda_{1}^{3}}{\lambda_{2}^{4}}h_{441}^{2}-\frac{\lambda_{1}^{3}}{\lambda_{2}^{4}}h_{444}^{2}+\frac{2}{\lambda_{1}-\lambda_{2}}h_{123}^{2}\\
&-\frac{2}{\lambda_{1}+\lambda_{2}}h_{234}^{2}-\frac{\lambda_{2}\lambda_{1}^{2}}{2(\lambda_{1}^{2}-\lambda_{2}^{2})},\\
\end{aligned}\\
\end{aligned}
\end{equation}

\begin{equation}
\begin{aligned}
&h_{1144}=
\begin{aligned}[t]
&\frac{1}{\lambda_{2}}h_{442}^{2}-\frac{1}{\lambda_{2}}h_{443}^{2}-\frac{2(\lambda_{1}^{4}+2\lambda_{1}^{2}\lambda_{2}^{2}-\lambda_{2}^{4})}{\lambda_{1}\lambda_{2}^{2}(\lambda_{1}^{2}-\lambda_{2}^{2})}h_{444}^{2}
-\frac{2}{\lambda_{1}-\lambda_{2}}h_{124}^{2}\\
&-\frac{2}{\lambda_{1}+\lambda_{2}}h_{134}^{2}+\frac{\lambda_{1}\lambda_{2}^{2}}{2(\lambda_{1}^{2}-\lambda_{2}^{2})}.\\
\end{aligned}\\
\end{aligned}
\end{equation}
By (3.10), (3.26) and (3.29), we conclude that
\begin{equation}
\frac{\lambda_{1}^{4}+2\lambda_{1}^{2}\lambda_{2}^{2}-\lambda_{2}^{4}}{\lambda_{1}^{2}\lambda_{2}^{2}}\left(h_{441}^{2}+h_{444}^{2} \right)+2\left( h_{124}^{2}+h_{134}^{2}\right)
=\frac{\lambda_{2}^{2}}{2}+\lambda_{1}^{2}\left(\lambda_{1}^{2}-\lambda_{2}^{2}\right).
\end{equation}
Similarly, by (3.10), (3.27) and (3.28), we conclude that
\begin{equation}
\frac{\lambda_{1}^{4}+2\lambda_{1}^{2}\lambda_{2}^{2}-\lambda_{2}^{4}}{\lambda_{1}^{2}\lambda_{2}^{2}}\left(h_{442}^{2}+h_{443}^{2} \right)+2\left( h_{123}^{2}+h_{234}^{2}\right)
=\frac{\lambda_{1}^{2}}{2}-\lambda_{2}^{2}\left(\lambda_{1}^{2}-\lambda_{2}^{2}\right).
\end{equation}
Since $S=$constant, from (2.7), we have
\begin{equation}
2\left(2+3\frac{\lambda_{1}^{4}}{\lambda_{2}^{4}}  \right)\left( h_{441}^{2}+h_{444}^{2}\right)+2\left(3+2\frac{\lambda_{1}^{4}}{\lambda_{2}^{4}}  \right)\left( h_{442}^{2}+h_{443}^{2}\right)+6\left( h_{124}^{2}+h_{134}^{2}+h_{123}^{2}+h_{234}^{2}\right)=S(S-1).
\end{equation}
Combining (3.30), (3.31) and (3.32), we obtain
\begin{equation}
\begin{aligned}
&\frac{3\lambda_{1}^{4}\lambda_{2}^{2}+2\lambda_{1}^{2}\lambda_{2}^{4}-3\lambda_{2}^{6}-6\lambda_{1}^{6}}{\lambda_{1}^{2}\lambda_{2}^{4}}\left( h_{441}^{2}+h_{444}^{2}\right)+
\frac{3\lambda_{1}^{2}\lambda_{2}^{4}+2\lambda_{1}^{4}\lambda_{2}^{2}-3\lambda_{1}^{6}-6\lambda_{2}^{6}}{\lambda_{2}^{6}}\left( h_{442}^{2}+h_{443}^{2}\right)\\
=&\frac{7}{2}\left(\lambda_{1}^{2}+\lambda_{2}^{2} \right)+3\left(\lambda_{1}^{2}-\lambda_{2}^{2} \right)^{2}-4\left(\lambda_{1}^{2}+\lambda_{2}^{2} \right)^{2}.
\end{aligned}
\end{equation}
By Lemma 2.1, considering the invariant $A$ first, we have
\begin{equation}
\begin{aligned}
A=&\frac{1}{5}(3Sf_{4}-2f_{4}-f_{3}^{2})\\
=&2\lambda_{1}^{2}\left(2+\frac{\lambda_{1}^{4}}{\lambda_{2}^{4}}+2\frac{\lambda_{1}^{2}}{\lambda_{2}^{2}} \right)\left( h_{441}^{2}+h_{444}^{2}\right)+2
\left(2\lambda_{1}^{2}+\lambda_{2}^{2}+2\frac{\lambda_{1}^{4}}{\lambda_{2}^{2}} \right)\left( h_{442}^{2}+h_{443}^{2}\right)\\
&+2\left(\lambda_{1}^{2}+2\lambda_{2}^{2}\right)\left( h_{123}^{2}+h_{234}^{2}\right)+2\left(2\lambda_{1}^{2}+\lambda_{2}^{2}\right)\left( h_{124}^{2}+h_{134}^{2}\right).
\end{aligned}
\end{equation}
Combining (3.30), (3.31) and (3.34), we obtain
\begin{equation}
\begin{aligned}
&\frac{2\lambda_{1}^{8}+\lambda_{2}^{8}+2\lambda_{1}^{6}\lambda_{2}^{2}-\lambda_{1}^{4}\lambda_{2}^{4}}{\lambda_{1}^{2}\lambda_{2}^{4}}\left( h_{441}^{2}+h_{444}^{2}\right)+
\frac{2\lambda_{2}^{8}+\lambda_{1}^{8}+2\lambda_{1}^{2}\lambda_{2}^{6}-\lambda_{1}^{4}\lambda_{2}^{4}}{\lambda_{2}^{6}}\left( h_{442}^{2}+h_{443}^{2}\right)\\
=&\frac{2}{5}\left(\lambda_{1}^{6}+\lambda_{2}^{6} \right)+\frac{22}{5}\left(\lambda_{1}^{4}\lambda_{2}^{2}+\lambda_{1}^{2}\lambda_{2}^{4} \right)-\frac{13}{10}\left(\lambda_{1}^{4}+\lambda_{2}^{4} \right)-2\lambda_{1}^{2}\lambda_{2}^{2}.
\end{aligned}
\end{equation}
Now, (3.33) and (3.35) can be considered as an linear equation system of $h_{441}^{2}+h_{444}^{2}$ and $h_{442}^{2}+h_{443}^{2}$. By standard linear algebra theory, we obtain
\begin{equation}
h_{441}^{2}+h_{444}^{2}=\frac{D^{(1)}}{D},\ \ h_{442}^{2}+h_{443}^{2}=\frac{D^{(2)}}{D},
\end{equation}
where $D=-\frac{5(\lambda_{1}^{2}-\lambda_{2}^{2})^{3}(\lambda_{1}^{2}+\lambda_{2}^{2})^{2}}{\lambda_{2}^{8}}$,
\begin{equation*}
\begin{aligned}
D^{(1)}=&-\frac{1}{\lambda_{2}^{6}}\left(\frac{1}{5}\lambda_{1}^{12}-\frac{8}{5}\lambda_{1}^{10}\lambda_{2}^{2}+\frac{16}{5}\lambda_{1}^{8}\lambda_{2}^{4}
-\frac{32}{5}\lambda_{1}^{6}\lambda_{2}^{6}-\frac{83}{5}\lambda_{1}^{4}\lambda_{2}^{8}-\frac{24}{5}\lambda_{1}^{2}\lambda_{2}^{10}    \right.\\
&\left.\frac{2}{5}\lambda_{2}^{12}-\frac{2}{5}\lambda_{1}^{10}+\frac{1}{10}\lambda_{1}^{8}\lambda_{2}^{2}+\frac{1}{2}\lambda_{1}^{6}\lambda_{2}^{4} +
\frac{43}{10}\lambda_{1}^{4}\lambda_{2}^{6}+\frac{59}{10}\lambda_{1}^{2}\lambda_{2}^{8}-\frac{4}{5}\lambda_{2}^{10} \right)
\end{aligned}
\end{equation*}
and
\begin{equation*}
\begin{aligned}
D^{(2)}=&\frac{1}{\lambda_{1}^{2}\lambda_{2}^{4}}\left(\frac{2}{5}\lambda_{1}^{12}-\frac{24}{5}\lambda_{1}^{10}\lambda_{2}^{2}-\frac{83}{5}\lambda_{1}^{8}\lambda_{2}^{4}
-\frac{32}{5}\lambda_{1}^{6}\lambda_{2}^{6}+\frac{16}{5}\lambda_{1}^{4}\lambda_{2}^{8}-\frac{8}{5}\lambda_{1}^{2}\lambda_{2}^{10}    \right.\\
&\left.\frac{1}{5}\lambda_{2}^{12}-\frac{4}{5}\lambda_{1}^{10}+\frac{59}{10}\lambda_{1}^{8}\lambda_{2}^{2}+\frac{43}{10}\lambda_{1}^{6}\lambda_{2}^{4} +
\frac{1}{2}\lambda_{1}^{4}\lambda_{2}^{6}+\frac{1}{10}\lambda_{1}^{2}\lambda_{2}^{8}-\frac{2}{5}\lambda_{2}^{10} \right).
\end{aligned}
\end{equation*}
Second, considering the invariant $B$, by Lemma 2.1 we have
\begin{equation}
\begin{aligned}
B=&\frac{1}{5}(2f_{3}^{2}-Sf_{4}-f_{4})\\
=&2\frac{\lambda_{1}^{4}}{\lambda_{2}^{4}}\left( h_{441}^{2}+h_{444}^{2}\right)+2\lambda_{1}^{2}\left( h_{442}^{2}+h_{443}^{2}\right)\\
&-2\lambda_{1}^{2}\left( h_{124}^{2}+h_{134}^{2}\right)-2\lambda_{2}^{2}\left( h_{123}^{2}+h_{234}^{2}\right).
\end{aligned}
\end{equation}
Combining (3.30), (3.31) and (3.37), we obtain
\begin{equation}
\begin{aligned}
&\frac{3\lambda_{1}^{4}+2\lambda_{1}^{2}\lambda_{2}^{2}-\lambda_{2}^{4}}{\lambda_{2}^{2}}\left( h_{441}^{2}+h_{444}^{2}\right)+
\frac{\lambda_{1}^{2}\left( 3\lambda_{2}^{4}+2\lambda_{1}^{2}\lambda_{2}^{2}-\lambda_{1}^{4}\right)}{\lambda_{2}^{4}}\left( h_{442}^{2}+h_{443}^{2}\right)\\
=&-\frac{9}{5}\left(  \lambda_{1}^{6}+\lambda_{2}^{6}+\lambda_{1}^{4}\lambda_{2}^{2} +\lambda_{1}^{2}\lambda_{2}^{4} \right)-\frac{2}{5}\left( \lambda_{1}^{4}+\lambda_{2}^{4}\right)    +\lambda_{1}^{2}\lambda_{2}^{2}.
\end{aligned}
\end{equation}
Putting (3.36) into (3.38), after calculating, we find that the lower order terms cancel out. It's amazing! More precisely, we get that
\begin{equation*}
\left(\lambda_{1}^{2}-\lambda_{2}^{2} \right)^{3}\left(\lambda_{1}^{2}+\lambda_{2}^{2} \right)^{3}\left(\lambda_{1}^{4}-\lambda_{1}^{2}\lambda_{2}^{2}+\lambda_{2}^{4}  \right)=0.
\end{equation*}
It's a contradiction with $\lambda_{1}>\lambda_{2}>0$.

\textbf{Case \rmnum{3}: $\lambda_{1}=-\lambda_{4}\neq 0, \lambda_{2}=\lambda_{3}=0$.}\\
Let's begin with the first derivatives of $H$, $S$ and $f_{3}$. By (3.2), (3.4) and (3.6), we get
\begin{equation}
\left\{\begin{aligned}
&h_{111}+h_{221}+h_{331}+h_{441}=\lambda_{1}\langle X(P),e_{1}\rangle,\\
&\lambda_{1}h_{111}-\lambda_{1}h_{441}=0,\\
&\lambda_{1}^{2}h_{111}+\lambda_{1}^{2}h_{441}=0,
\end{aligned}\right.
\end{equation}

\begin{equation}
\left\{\begin{aligned}
&h_{112}+h_{222}+h_{332}+h_{442}=0,\\
&\lambda_{1}h_{112}-\lambda_{1}h_{442}=0,\\
&\lambda_{1}^{2}h_{112}+\lambda_{1}^{2}h_{442}=0,
\end{aligned}\right.
\end{equation}

\begin{equation}
\left\{\begin{aligned}
&h_{113}+h_{223}+h_{333}+h_{443}=0,\\
&\lambda_{1}h_{113}-\lambda_{1}h_{443}=0,\\
&\lambda_{1}^{2}h_{113}+\lambda_{1}^{2}h_{443}=0,
\end{aligned}\right.
\end{equation}
and
\begin{equation}
\left\{\begin{aligned}
&h_{114}+h_{224}+h_{334}+h_{444}=-\lambda_{1}\langle X(P),e_{4}\rangle,\\
&\lambda_{1}h_{114}-\lambda_{1}h_{444}=0,\\
&\lambda_{1}^{2}h_{114}+\lambda_{1}^{2}h_{444}=0.
\end{aligned}\right.
\end{equation}
The above equation systems give us
\begin{equation}
\left\{\begin{aligned}
&h_{111}=h_{112}=h_{113}=h_{114}=h_{441}=h_{442}=h_{443}=h_{444}=0,\\
&\langle X(P),e_{1}\rangle=\frac{1}{\lambda_{1}}\left(h_{221}+h_{331}\right), \langle X(P),e_{4}\rangle=-\frac{1}{\lambda_{1}}\left(h_{224}+h_{334}\right),\\
&h_{222}=-h_{332}, h_{223}=-h_{333}.
\end{aligned}\right.
\end{equation}
Now, we will proceed to the second derivatives of $S$, $f_{3}$ and $f_{4}$. Combining (3.5) and (3.43), we have
\begin{equation}
\left\{\begin{aligned}
&\lambda_{1}(h_{1111}-h_{4411})=-h_{221}^{2}-h_{331}^{2}-2h_{123}^{2}-2h_{124}^{2}-2h_{134}^{2},\\
&\lambda_{1}(h_{1122}-h_{4422})=
\begin{aligned}[t]
&-h_{222}^{2}-h_{332}^{2}-2h_{122}^{2}-2h_{132}^{2}\\
&-2h_{142}^{2}-2h_{232}^{2}-2h_{242}^{2}-2h_{342}^{2},\\
\end{aligned}\\
&\lambda_{1}(h_{1133}-h_{4433})=
\begin{aligned}[t]
&-h_{223}^{2}-h_{333}^{2}-2h_{123}^{2}-2h_{133}^{2}\\
&-2h_{143}^{2}-2h_{233}^{2}-2h_{243}^{2}-2h_{343}^{2},\\
\end{aligned}\\
&\lambda_{1}(h_{1144}-h_{4444})=-h_{224}^{2}-h_{334}^{2}-2h_{124}^{2}-2h_{134}^{2}-2h_{234}^{2}.
\end{aligned}\right.
\end{equation}
Combining (3.7) and (3.43), we have
\begin{equation}
\left\{\begin{aligned}
&\lambda_{1}^{2}(h_{1111}+h_{4411})=2\lambda_{1}(h_{124}^{2}+h_{134}^{2}),\\
&\lambda_{1}^{2}(h_{1122}+h_{4422})=2\lambda_{1}(h_{224}^{2}+h_{234}^{2}-h_{122}^{2}-h_{123}^{2}),\\
&\lambda_{1}^{2}(h_{1133}+h_{4433})=2\lambda_{1}(h_{334}^{2}+h_{234}^{2}-h_{133}^{2}-h_{123}^{2}),\\
&\lambda_{1}^{2}(h_{1144}+h_{4444})=-2\lambda_{1}(h_{124}^{2}+h_{134}^{2}).\\
\end{aligned}\right.
\end{equation}
Combining (3.9) and (3.43), we have
\begin{equation}
\left\{\begin{aligned}
&\lambda_{1}^{3}(h_{1111}-h_{4411})=-2\lambda_{1}^{2}(h_{124}^{2}+h_{134}^{2}),\\
&\lambda_{1}^{3}(h_{1122}-h_{4422})=-2\lambda_{1}^{2}(h_{122}^{2}+h_{123}^{2}+h_{124}^{2}+h_{224}^{2}+h_{234}^{2}),\\
&\lambda_{1}^{3}(h_{1133}-h_{4433})=-2\lambda_{1}^{2}(h_{123}^{2}+h_{133}^{2}+h_{134}^{2}+h_{234}^{2}+h_{334}^{2}),\\
&\lambda_{1}^{3}(h_{1144}-h_{4444})=-2\lambda_{1}^{2}(h_{124}^{2}+h_{134}^{2}).\\
\end{aligned}\right.
\end{equation}
Combining the first equality of (3.44), the first equality of (3.45) and the first equality of (3.46), we get
\begin{equation}
\left\{\begin{aligned}
&h_{122}=h_{123}=h_{133}=0,\\
&h_{1111}=0, h_{4411}=\frac{2}{\lambda_{1}}(h_{124}^{2}+h_{134}^{2}).
\end{aligned}\right.
\end{equation}
Combining the fourth equality of (3.44), the fourth equality of (3.45) and the fourth equality of (3.46), we get
\begin{equation}
\left\{\begin{aligned}
&h_{224}=h_{334}=h_{234}=0,\\
&h_{4444}=0, h_{1144}=-\frac{2}{\lambda_{1}}(h_{124}^{2}+h_{134}^{2}).
\end{aligned}\right.
\end{equation}
Combining (3.47), (3.48), the second equality of (3.44), the second equality of (3.45) and the second equality of (3.46), we get
\begin{equation}
\left\{\begin{aligned}
&h_{222}=h_{332}=h_{223}=h_{333}=0,\\
&h_{1122}=-\frac{1}{\lambda_{1}}h_{124}^{2}, h_{4422}=\frac{1}{\lambda_{1}}h_{124}^{2}.
\end{aligned}\right.
\end{equation}
Combining (3.47), (3.48), the third equality of (3.44), the third equality of (3.45) and the third equality of (3.46), we get
\begin{equation}
h_{333}=h_{223}=h_{332}=h_{222}=0¡£
\end{equation}
Therefore all $h_{ijk}=0$ except $h_{124}$ and $h_{134}$ in the sense of the Codazzi equations (2.2). From (3.47) and (3.48), we get
\begin{equation}
h_{1144}-h_{4411}=-\frac{4}{\lambda_{1}}(h_{124}^{2}+h_{134}^{2}).
\end{equation}
Combining (3.51) and (3.10), we have
\begin{equation}
2(h_{124}^{2}+h_{134}^{2})=\lambda_{1}^{4}.
\end{equation}
Since $S$=constant, (2.7) gives us
\begin{equation}
S(S-1)=6(h_{124}^{2}+h_{134}^{2}).
\end{equation}
Because of $S=2\lambda_{1}^{2}$, combining (3.52) and (3.53), we get
\begin{equation}
\lambda_{1}=\sqrt{2},\ \ h_{124}^{2}+h_{134}^{2}=2,\ \ S=4.
\end{equation}
Moreover, we have $f_{4}=8$. By using of Proposition 2.5, we get $H\equiv 0$ on $M$. Then at $P$, (3.3) gives $0=\nabla_{1}\nabla_{1}H=\lambda_{1}$. It's a contradiction.

We conclude that $S=0$ and $M$ is isometric to $\mathbb{R}^{4}$.

\section{Proof of Theorem 1.4 and Theorem 1.5}
\emph{Proof of Theorem 1.4.} From (2.7), we know that $S=0$ or $S\geq 1$. If $S=0$, then $M$ is isometric to $\mathbb{R}^{4}$. Now, we assume that $S\geq 1$. By Lemma 2.4, it's easy to see that $\inf |H|=0$ since $f_{3}=0$. If $H$ vanishes somewhere, then from Section 3 we know that this case can not occur. Hence $H$ doesn't change sign on $M$. We may assume that $H>0$ on $M$ and $\inf H=0$ without loss of generality. Similar to the proof of Proposition 2.5, we can apply the generalized maximum principle (Lemma 2.3) for $\mathcal{L}$-operator to $-H$. Calculate by taking the limit just like the proof of Proposition 2.5 in all the calculations almost the same as Section 3. We can also conclude that this case can not occur. Thus, $S=0$, i.e. $M$ is isometric to $\mathbb{R}^{4}$. The theorem follows.\hfill$\Box$\\
\emph{Proof of Theorem 1.5.} If $\overline{H}=\inf |H|=0$, that is $f_{3}=0$. By Theorem 1.4, we know that $M$ is isometric to $\mathbb{R}^{4}$. If $\overline{H}=\inf |H|>0$, by Lemma 2.4 and $f_{3}=\frac{3}{4}\overline{H}S-\frac{\overline{H}^{3}}{8}$, we conclude that $M$ is isometric to one of $S^{2}(\sqrt{2})\times \mathbb{R}^{2}$ and $S^{4}(2)$. The theorem follows.\hfill$\Box$\\

\end{document}